\documentclass[12pt]{amsart}
\usepackage{amsthm,amsmath,amsxtra,amscd,amssymb,xypic,color}

\renewcommand{\Im}{{\operatorname{Im}\, }}
\newcommand{\SL}{{\operatorname{SL}}}
\newcommand{\CC}{{\mathbb{C}}}

\newcommand{\HH}{{\mathbb{H}}}

\newcommand{\PP}{{\mathbb{P}}}

\newcommand{\RR}{{\mathbb{R}}}
\newcommand{\ZZ}{{\mathbb{Z}}}

\newcommand{\calL}{{\mathcal L}}
\newcommand{\calM}{{\mathcal M}}

\newcommand{\calP}{{\mathcal P}}

\newcommand{\calS}{{\mathcal S}}
\newcommand{\calK}{{\mathcal K}}

\theoremstyle{plain}
\newtheorem{thm}{Theorem}
\newtheorem{lm}[thm]{Lemma}
\newtheorem{prop}[thm]{Proposition}
\newtheorem{cor}[thm]{Corollary}

\theoremstyle{definition}
\newtheorem{definition}{Definition}
\newtheorem{rem}[definition]{Remark}

\begin{document}
\title[RN differentials and the elliptic Calogero-Moser system]{Real-normalized differentials and the elliptic Calogero-Moser system}
\author{Samuel Grushevsky}
\address{Mathematics Department, Stony Brook University, Stony Brook, NY 11794-3651, USA}
\email{sam@math.stonybrook.edu}
\thanks{Research of the first author was supported in part by National Science Foundation under the grant DMS-12-01369.}
\author{Igor Krichever}
\address{Columbia University, New York, USA, and Kharkevich Institute for Information Transmission Problems, and Research University ``Higher School of Economics'', and
Landau Institute for Theoretical Physics, Moscow, Russia}
\email{krichev@math.columbia.edu}
\thanks{Research of the second author was supported in part by Russian Fund for fundamental research under the grants 14-01-00012 and 13-01-12469.}


\maketitle

\begin{abstract}
In our recent works \cite{grkrwhitham}, \cite{grkrdeform} we have used meromorphic differentials on Riemann surfaces all of whose periods are real to study the geometry of the moduli spaces of Riemann surfaces. In this paper we survey the relevant constructions and show how they are related to and motivated by the spectral theory of the elliptic Calogero-Moser integrable system.
\end{abstract}

\section*{Introduction}
This paper is part of a series studying the geometry of the moduli space of curves using meromorphic differentials with real periods (so-called real-normalized differentials, see below). These differentials were introduced in \cite{kr86} in full generality by the second author, while the idea for some special case goes back at least to Maxwell. More recently, in \cite{grkrwhitham} we have used these differentials to give a direct proof of Diaz' theorem \cite{diaz} on the dimension of complete subvarieties of the moduli space of curves. Further, in \cite{grkrdeform} we have described the local infinitesimal structure of the foliation on the moduli space defined using the periods of a real-normalized differential. In \cite{grkrno} together with Chaya Norton we studied in detail the behavior of real-normalized differentials under degeneration. In the current paper we first review the main framework of this setup and then present part of the motivation for the constructions using two real-normalized differentials simultaneously. Namely, we show that if one takes the two foliations on the moduli space corresponding to two different real-normalized differentials, then some intersections of their leaves are in fact algebraic subvarieties, which are equal to the suitable loci of spectral curves of the elliptic Calogero-Moser completely integrable system (see theorem \ref{thm:main} for a precise statement). This result provides motivation for some of the conjectures that we make in \cite{grkrdeform}. In \cite{grkrcusps} we will further use this motivation and the degeneration techniques of \cite{grkrno} to obtain results on cusps of plane curves and on the cohomology of the moduli space of curves.

\section{Real-normalized differentials and foliations by absolute periods}
\begin{definition}
Let $\calM:=\calM_{g,1}(1)$ denote the moduli space of genus $g$ Riemann surfaces $C$ with one marked point $p\in C$ and a $1$-jet of a local coordinate at $p$: that is to say a local coordinate $z$ on a small analytic neighborhood of $p$ in $C$, where we identify $z$ and $z'$ iff $z'=z+O(z^2)$. Algebro-geometrically, $\calM$ is the total space of the relative tangent bundle at $p$ over $\calM_{g,1}$ with the zero section removed (as the local coordinate must be non-degenerate). We denote a point $(C,p,z)\in\calM$ by $X$.
\end{definition}

\begin{definition}
A meromorphic differential $\eta$ on a Riemann surface $C$ is called {\em real-normalized} if all its periods are real, i.e.~if for any closed loop $\gamma\in H_1(C,\ZZ)$ we have $\int_\gamma\eta\in\RR$.
\end{definition}
We notice in particular that the residue of a real-normalized differential at any point must be purely imaginary, so that its integral over a small loop around that point is real. From the positive-definiteness of the imaginary part of the period matrix of a Riemann surface it follows that the only {\em holomorphic} real-normalized differential is identically zero. Since any $\RR$-linear combination of real-normalized differentials is real-normalized, it follows that the singularities of a real-normalized differential determine it uniquely, and in particular we have
\begin{prop}
For any $X\in\calM$ there exists a unique meromorphic real-normalized differential $\Psi=\Psi_X$ on $C$ whose only singularity is a double pole at $p$, where it singular part is equal to $dz/z^2$ in the chosen jet.
\end{prop}
We refer to \cite{grkrwhitham} for a detailed proof and more comments and references on the history of real-normalized differentials. For further use, we will also denote $\Psi'=\Psi'_X$ the real-normalized differential with the unique singularity being a double pole at $p$ of the form $idz/z^2$.

More generally, a unique real-normalized meromorphic differential exists for any prescribed collection of singular parts with zero residues. If such a real-normalized differential $\eta$ is exact, $\eta=d f$, then $f:C\to\PP^1$ is a meromorphic function with prescribed poles at the singularities of $\eta$. This is exactly to say that $f$ exhibits $C$ as a cover of $\PP^1$ with prescribed ramification, and the space of such $f$ is precisely the Hurwitz space. The Lyashko-Looijenga coordinates are then defined to be the critical values of $f$, that is the values of $f$ at its critical points, the points where $\eta=df=0$. It is known that Lyashko-Looijenga coordinates give local coordinates on the Hurwitz space (see \cite{lazv} for an exposition of the theory), and similarly one can define local coordinates using real-normalized differentials in general. This construction is also a special (real-normalized) case of the generalization to meromorphic differentials of the well-known construction for holomorphic differentials, see \cite{zorichsurvey} for a survey of these ideas.

For any $X\in\calM$ we denote by $q_1,\ldots,q_{2g}$ the zeroes of $\Psi_X$, written with multiplicity.
\begin{definition}
If $X\in\calM$ is such that all $q_i$ are distinct, we choose a sufficiently small analytic neighborhood of $X$ where the zeroes remain distinct, and moreover over which we can choose a continuously varying symplectic bases $A_1,\ldots,A_g,B_1,\ldots,B_g$ of $H_1(C,\ZZ)$ (that is to say, $A_i\cdot B_j=\delta_{ij}$, while $A_i\cdot A_j=B_i\cdot B_j=0$). The {\em absolute periods} of $\Psi$ are the integrals $\alpha_i:=\int_{A_i}\Psi\in\RR$ and $\beta_i:=\int_{B_i}\Psi\in\RR$. These are real-analytic functions on such a neighborhood of $X$. The {\em relative periods} of $\Psi$ are the integrals $\int_*^{q_1}\Psi,\ldots,\int_*^{q_{2g}}\Psi$, where the paths of integration are chosen not to intersect $A_i$ or $B_i$, and the basepoint $*$ is chosen so that the sum of all relative periods is equal to zero. We note that the full collection of absolute and relative periods is in fact the pairing of $\Psi$ with a basis of the relative homology of the punctured surface $H_1(C\setminus\lbrace p\rbrace,\lbrace q_1,\ldots,q_{2g}\rbrace)$, see \cite{kallakorotkin}.
\end{definition}
It turns out that these periods give local coordinates:
\begin{prop}[\cite{krph}, see also \cite{grkrwhitham}]
The collection of absolute and relative periods gives real-analytic local coordinates, with values in $\RR^{2g}\times\CC^{2g-1}$, near any $X\in\calM$ where all zeroes of $\Psi$ are distinct, which we call the {\em period coordinates}.
\end{prop}
This statement should be viewed as a generalization of the fact that we have Lyashko-Looijenga
coordinates on the Hurwitz space. Indeed, the proposition above holds in full generality for any prescribed singular parts, and in that case if we restrict to the locus where all the absolute periods are zero --- so that $\Psi$ is exact --- the relative periods are precisely the Lyashko-Looijenga coordinates on the corresponding Hurwitz space.
\begin{rem}
The above definition of the period coordinates can be generalized to define local coordinates near a point of $\calM$ where $\Psi$ has multiple zeroes --- in that case locally a zero of $\Psi$ of multiplicity $k$ may be perturbed to become at most $k$ zeroes. In a neighborhood, we then choose paths to these at most $k$ zeroes that coincide outside a small neighborhood of zero (basically going to where the $k$-multiple zero was, and then perturbing the ends of the paths to go to the individual zeroes), and then instead of individual relative periods consider the symmetric function of the corresponding up to $k$ relative periods, counted with multiplicity (so that at the original point we have symmetric functions of a $k$-tuple of equal values). This is discussed in more detail \cite{grkrwhitham}, and provides a viewpoint on the coordinates on various strata of holomorphic differentials with prescribed configurations of zeroes studied in Teichm\"uller dynamics --- but will not play in a role in the current paper.
\end{rem}
A crucial observation that made the results of \cite{grkrwhitham} possible is that while the values of absolute periods are not well-defined on $\calM$, as they depend on a choice of a symplectic homology basis, the condition that they are locally constant is well-defined.
\begin{definition}
For any $X\in\calM$ we define a {\em leaf of a (big) foliation} $\calL_X$ passing through $X$ to locally be the locus of points in $\calM$ where the absolute periods are constant and equal to those on $X$. As noted above, this condition is independent of the choice of the symplectic homology basis, and we thus define a foliation $\calL$ of $\calM$ to consist of these leaves (the analogous foliation for holomorphic differentials is sometimes called the REL foliation or the absolute period foliation in Teichm\"uller dynamics. We note that relative periods give local holomorphic coordinates on each leaf, and we thus have a tangentially complex foliation on $\calM$: the tangent space to $\calL_X$ at any $X$ is complex. In particular note that all leaves of $\calL$ are smooth immersed complex submanifolds of $\calM$ of complex codimension $g$.
\end{definition}
Note that the leaves of $\calL$ are only defined locally, and their global geometry in general can be extremely complicated, with the closure in the Deligne-Mumford compactification being especially badly behaved, see \cite{mcmullen} for a study of a similar situation for holomorphic differentials, and \cite{grkrno} for a study of degenerations and the behavior of local coordinates in our setting. However, note that (for more general singularities) a leaf of $\calL$ corresponding to all absolute periods being zero is the suitable Hurwitz space, and is algebraic. More generally, if say all absolute periods $\alpha_i,\beta_i$ on a leaf $\calL$ are rational (or, still more generally, generate an additive subgroup of $\RR$ in which $0$ is an isolated point), then $\calL$ is an embedded (as opposed to only immersed) submanifold of $\calM$. Indeed, if two points $X_1,X_2$ that are close in $\calM$ both lie in $\calL$, in a neighborhood of $X_1$ choose a continuously varying basis for $H_1(X,\ZZ)$. Then the periods of $\Psi$ over this basis must lie in the same discrete subgroup of $\RR$ for both $X_1$ and $X_2$, and thus all the periods of $\Psi$ over this basis of cycles must be the same for $X_1$ and $X_2$. This means that $X_1$ and $X_2$ lie on the same connected component of the intersection of $\calL$ with a small neighborhood of $X_1$, which is to say that the intersection of $\calL$ with a small neighborhood in $\calM$ of any its point is connected, and thus $\calL\subset\calM$ is then embedded.

\smallskip
All of the above constructions were completely general, and developed in \cite{grkrwhitham} in full generality for arbitrary prescribed singularities. We now take into account the specifics of our situation, where we naturally have {\em two} real-normalized differentials $\Psi$ and $\Psi'$. We first have the easy observation:
\begin{lm}\label{lm:R2d}
For any $X=(C,p,z)\in\calM$ any real-normalized differential on $C$ with a double pole at $p$ is an $\RR$-linear combination of $\Psi_X$ and $\Psi'_X$.
\end{lm}
\begin{proof}
Indeed, suppose such a real-normalized differential $\eta$ has a singular part $(a+bi)dz/z^2$, for $a,b\in\RR$ (notice that $\eta$ cannot have a residue at its unique pole on a compact Riemann surface). Then the differential $\eta-a\Psi-b\Psi'$ is a holomorphic real-normalized differential on $C$, and thus must be zero.
\end{proof}
It thus turns out that some properties of the pair $(\Psi,\Psi')$ are independent of the choice of the local coordinate, and thus make sense on $\calM_{g,1}$ --- this will be used in \cite{grkrcusps} to study the homology classes of the loci in $\calM_{g,1}$ where $\Psi$ and $\Psi'$ have a prescribed number of common zeroes. We thus overlap the above structures defined by $\Psi$ and $\Psi'$.
\begin{definition}
For any $X\in\calM$, let $\calL_X$ be the leaf of the foliation defined above. Let $\calL'_X$ be the leaf of the analogous foliation on $\calM$ defined using the period coordinates associated to $\Psi'$. We then let $\widehat\calS_X=\calL_X\cap\calL'_X\subset\calM$ be the intersection of these two leaves. From the lemma above we see that the condition of local constancy of absolute periods of {\em both} $\Psi$ and $\Psi'$ implies the constancy of periods of the real-normalized differential with any fixed double pole. That is to say, $\widehat\calS_X$ does not depend on the choice of a local coordinate, and is a preimage of some locus $\calS_X\subset\calM_{g,1}$.
\end{definition}

By a slight abuse of language, we will call $\calS$ the ``small'' foliation on $\calM$, and call $\calS_X$ its leaves. Here we use the term foliation loosely: indeed, one of the main conjectures of \cite{grkrdeform} is precisely the statement that all the leaves $\calS_X$ are smooth (and of the same dimension). Thus what we mean by a foliations is only that there exists a subvariety $\calS_X$ through every point of $\calM_{g,1}$, and distinct such sets do not intersect.
However, the main result of this paper is precisely the statement that there exists an everywhere dense collection of leaves $\calS_X$ that are in fact smooth algebraic subvarieties of $\calM$. To prove this, we identify these leaves as the loci of suitable spectral curves of the elliptic Calogero-Moser integrable systems, and thus {\em any} leaf $\calS_X$ can be interpreted as a suitable perturbation of the spectral curves elliptic Calogero-Moser system, which motivates our conjectures in \cite{grkrdeform}.

\section{The Calogero-Moser locus via real-normalized differentials}
\begin{definition}
We define the Calogero-Moser locus $\calK_g\subset\calM_{g,1}$ to be the locus of all $(C,p)$ for which
there exist two $\RR$-linearly independent differentials of the second kind $\Phi_1,\Phi_2\in
H^0(C,K_C+2p)$ with all periods {\it integer}.
\end{definition}
We note that in particular $\Phi_1$ and $\Phi_2$ are real-normalized, and since by lemma \ref{lm:R2d} the
space of real-normalized differentials with a unique double pole is two-dimensional over $\RR$, any
real-normalized differential with a unique double pole at $p$ is equal to $r_1\Phi_1+r_2\Phi_2$ for
some $r_1,r_2\in\RR$. We thus have
\begin{prop}
The Calogero-Moser locus $\calK_g\subset\calM_{g,1}$ is the union of all leaves of the small
foliation $\calS$ for which there exist $r_{1},r_{2},r_{1}',r_{2}'\in\RR$ such that
all (absolute) periods of $\Psi$ lie in $r_1{\mathbb Z}+r_2{\mathbb Z}$ and all periods of $\Psi'$ lie in $r_1'\ZZ+r_2'\ZZ$.
\end{prop}
\begin{rem}
We note that the locus $\calK_g$ is easily seen to be dense in $\calM_{g,1}$, as for example it includes the set of all curves for which all absolute periods of $\Psi$ and $\Psi'$ are rational (there are finitely many periods, so we can just choose $r_1=r_1'$ to be the the inverse of the largest denominator). Of course the locus $\calK_g$ consists of infinitely many leaves of the foliation $\calS$ corresponding to different $r$'s, and has infinitely many connected components --- distinguished at least by the least common multiple of the denominators of periods. Instead of working with all of $\calK_g$, we will thus first represent it as a countable union of loci corresponding to different degrees of the covers of the elliptic curve that is inherent in the picture.
\end{rem}
Indeed, note that if $\Phi_1$ and $\Phi_2$ have integer periods, then so does any their $\ZZ$-linear combination. Thus for $(C,p)\in\calK_g$
we choose $\Phi_1,\Phi_2$ generating the lattice in $H^0(C,K_C+2p)$ for which all periods are integer, and let
\begin{equation}\label{cmtau}
\tau:=\frac{\Phi_2}{\Phi_1}(p)
\end{equation}
(which means taking the ratio of the singular parts of $\Phi_2$ and $\Phi_1$ at $p$).
Since $\Phi_i$ are $\RR$-linearly independent, $\Im \tau \ne 0$, and by swapping
$\Phi_1$ and $\Phi_2$ if necessary we may assume that $\Im\tau>0$.
While $\tau$ depends on the choice of generators of the lattice, and is only well-defined up to an action of $\SL(2,\ZZ)$, note that the differential $\Phi_2-\tau \Phi_1$ is {\em holomorphic} on $C$, and all its periods lie in $\mathbb Z+\tau\mathbb Z$. Thus integrating this differential gives
a holomorphic map
\begin{equation}\label{cmcover}
z: C\to E:=\CC/\ZZ+\tau\ZZ;\qquad q\mapsto z(q):=\int_{p}^q (\Phi_2-\tau\Phi_1)
\end{equation}
We note that the isomorphism class of the elliptic curve $E$ and the map $z:C\to E$ do not
depend on the choice of $\Phi_1,\Phi_2$.
\begin{definition}
For $N\in\ZZ$ we denote $\calK_{g,N}\subset\calK_g$ the locus of Calogero-Moser curves for which the
degree of the map $z:C\to E$ is equal to $N$. We will see that $\calK_{g,N}$ is
empty if $N<g$ as will follow from our explicit parametrization of these loci; it can also be shown that for any $N\ge g$ the locus $\calK_{g,N}$ is in fact non-empty, as can be seen by studying suitable degenerations and then perturbing --- but we will not need this result here.
\end{definition}
Since the degree is an integer that depends continuously on the Calogero-Moser curve, it is locally constant
on $\calK_g$, and thus each $\calK_{g,N}$ is a union of some collection of connected
components of $\calK_g$. Analytically, $N$ can be computed as
\begin{equation}\label{cm3}
N=\sum_{k=1}^g\left(\int_{A_k}\Phi_2\int_{B_k}\Phi_1-\int_{B_k}\Phi_2
\int_{S_k}\Phi_1\right)=2\pi i\,{\rm res}_{p}\left(F_1 \Phi_2\right).
\end{equation}
In what follows we will always fix the generators $\Phi_1,\Phi_2$, or equivalently fix a basis of $H_1(E,\ZZ)\equiv \ZZ^2$. For any $\tau\in\HH/\SL(2,\ZZ)$ we then denote $\calK_{g,N}^\tau\subset \calK_{g,N}\subset\calK_g\subset\calM_{g,1}$ the
subset where $E=E_\tau$. 
Since the above constructions and definition only depend on the absolute periods of $\Psi_1$ and $\Psi_2$, we have
the following
\begin{prop}\label{prop:CMisleaves}
The locus $\calK_{g,N}^\tau$ is a union of leaves of the small foliation $\calS\subset\calM_{g,1}$; that is to say,
if any leaf of $\calS$ intersects $\calK_{g,N}^\tau$ (for $N\in\ZZ,\tau\in\HH/\SL(2,\ZZ)$ fixed), then this
leaf is contained in $\calK_{g,N}^\tau$.
\end{prop}
The main result of this paper is justifying the name ``Calogero-Moser'' for this locus, i.e.~the identification of $\calK_g$ as the locus of spectral curves of the elliptic Calogero-Moser system. Our main result is the following
\begin{thm}\label{thm:main}
The locus $\calK_{g,N}$ is the locus
of curves that are normalizations $\widetilde C_{cm}$ of spectral curves $C_{cm}$ of the
$N$-particle elliptic Calogero-Moser system (i.e.~of curves given by \eqref{r} below).
\end{thm}
Note that in particular this theorem implies that the loci $\calK_{g,N}^\tau$ and $\calK_{g,N},$ are all algebraic, as they arise from the algebro-geometric constructions associated to the elliptic Calogero-Moser system.

\section{The elliptic Calogero-Moser system}
\begin{definition}
The {\em elliptic Calogero-Moser (CM) system} introduced in \cite{calogero}
is a system of $N$ particles on an elliptic curve $E$ with
pairwise interactions. The phase space of this system is
$$\begin{aligned}
\calP_N:&=(\CC\times E)^{\times N}\setminus\{{\rm diagonals\ in\ }E\}\\
&=\left\{q_1,\ldots,q_N\in\CC, x_1,\ldots,x_N\in E, x_i\neq x_j\right\},
\end{aligned}
$$
where we think of the variables $x_i$ as the positions of the particles, and of $q_i$ as their momenta, lying in the cotangent space to $E$, which is trivial and identified with $\CC$.
The evolution of this system is a trajectory of a set of particles in the phase
space.

The {\em elliptic Calogero-Moser Hamiltonian} is the meromorphic function $H_2:\calP_N\to\CC$ given by
\begin{equation}\label{H}
H_2:=\frac{1}{2}\sum_{i=1}^N q_i^2 - 2 \sum_{i\ne j}\wp(x_i-x_j),
\end{equation}
where $\wp$ denotes the Weierstrass $\wp$-function on $E$.
\end{definition}
The Hamiltonian equations of motions are then
$$
 \dot x_i=-\frac{\partial H_2}{\partial q_i};\quad  \dot q_i=\frac{\partial H_2}{\partial x_i};
$$
where from now on the dot denotes the partial derivative $\partial/\partial t$ with respect to
time. These equations of motion determine the evolution of the system of particles completely starting from the given initial conditions.

In \cite{kr-cm} the second author showed that the equations of motion of the elliptic CM system
admit a Lax representation with ``elliptic spectral parameter $z$''.
This is to say that the Hamiltonian equations of motion above for the Hamiltonian $H_2$ are
equivalent to the matrix-valued differential equation  $\dot L=[L,M]$, where~$L=L(z)$ and~$M=M(z)$ are
$N\times N$ matrices depending on the point $z\in E$, given explicitly by
\begin{equation}\label{matrixL}
 L_{ii}(z)=\frac12\ q_i;\quad L_{ij}(z)=F(x_i-x_j,z) \quad {\rm for\ } i\neq j,
\end{equation}
and
\begin{equation}\label{matrixM}
M_{ii}(z)=\wp(z)-2\sum_{k\neq i}\wp(x_i-x_k);\quad
M_{ij}(z)=-2 F' (x_i-x_j ,z) ,
\end{equation}
with the function $F$ defined by
\begin{equation}
 F(x,z):=\frac{\sigma(z-x)}{\sigma(z) \sigma(x)} e^{\zeta(z)x},
\end{equation}
for $\zeta$ and $\sigma$ the standard Weierstrass elliptic functions, and where $F'$ denotes the derivative of $F$ with respect to $x$.
\begin{definition}
The {\em spectral curve} $C_{cm}$ of the elliptic CM system is the normalization
at the point $(k,z)=(\infty,0)$ of the closure in ${\mathbb P}^1\times E$ of the affine curve $C_{cm}^o\subset \CC\times(E\setminus\{0\})$
given by the equation
\begin{equation}\label{r}
 R(k,z):=\det (k\cdot I+L(z))=0,
\end{equation}
where $I$ is the identity matrix. For further use, we expand this determinant as a polynomial in powers of $k$,
denoting the coefficients $r_i(z)$, so that $R(k,z)=\sum_{i=0}^N r_i(z)k^{N-i}$, in particular with $r_0(z)=1$.
We note that $C_{cm}$ is singular at all singularities of $C_{cm}^o$, and denote $\widetilde C_{cm}$ the normalization of $C_{cm}$.
\end{definition}
\begin{rem}
It is easy to see that the Lax equation $\dot L=[M,L]$ directly implies that the characteristic  equation satisfied by the differential operator $L$ does not depend on $t$, i.e.~the spectral curve can be regarded as ``integrals of motion'' (is time-invariant).
The general algebro-geometric integration scheme of soliton systems based on a concept of the Baker-Akhiezer functions in fact establishes the one-to-one correspondence of the open sets of the phase space of the system and the Jacobian bundle over the family of the corresponding spectral curves. Under this correspondence the equations of motion of the system become the equations of the linear flow on the Jacobian.
\end{rem}
{}From the Riemann-Hurwitz formula it follows that the arithmetic genus of $C_{cm}$ is equal to $N$; thus the genus of its normalization $\widetilde{C}_{cm}$ is strictly less than $N$ if and only if $C_{cm}$ is singular.

{}From the explicit formula \eqref{matrixL} for $L(z)$ one sees that each $r_i(z)$ is a meromorphic function of $z\in E$ with a pole of order $i$ at $z=0$.
As shown in \cite{kr-cm}, near $z=0$ the polynomial $R(k,z)$
admits a factorization of the form
\begin{equation}\label{r1}
R(k,z)=\prod_{i=1}^N(k+a_iz^{-1}+h_i+O(z)),
\end{equation}
with $a_1=1-N$ and $a_i=1$ for $i>1$, for some $h_i\in \CC$. This implies that the closure $\overline{C_{cm}^o}\subset\PP^1\times E$ of $C_{cm}^o$ is obtained by adding one point $(\infty,0)$, at which $N-1$ branches of $\overline{C_{cm}^o}$ are tangent to each other (corresponding to $a_2=\ldots=a_N=1$), and one branch is transverse to them. Thus if we blow up the point $(\infty,0)\in\PP^1\times E$, on the strict transform of $\overline{C_{cm}^o}$ under this blowup we would have a smooth point $p$ corresponding to the first branch, and a point $p'$ contained in the $N-1$ branches will pass. Thus generically the partial normalization $C_{cm}$ of $\overline{C_{cm}^o}$ at $(\infty,0)$ is obtained by doing the second blowup at $p'$, and irrespective of this we have $\#\{C_{cm}\setminus C_{cm}^o\} =N$.

\begin{prop}[\cite{pd}]
For a fixed elliptic curve $E$ and a fixed integer $N$ the space of Calogero-Moser spectral curves $C_{cm}$ is equal to $\CC^N$,
i.e.~is parameterized by $N$ free complex parameters.
\end{prop}
These $N$ free complex parameters were found in \cite{pd} by observing that a polynomial $R(k,z)$ has a unique representation of the form
\begin{equation}\label{phdoker1}
 R(k,z):=f(k-\zeta(z),z),\quad
\end{equation}
where
\begin{equation}\label{ph}
f(\phi,z)=\frac{1}{\sigma(z)}\ \sigma\left(z+\frac{\partial}{\partial \phi}\right)H(\phi)=\frac{1}{\sigma(z)}\ \sum_{n=0}^N \frac{1}{ n!}\,\partial_z^{\,n} \sigma(z) \frac{\partial^n H}{\partial \phi^n}.
\end{equation}
and $H$ is the monic degree $N$ polynomial
$$
 H(\phi)=\phi^N+\sum_{i=0}^{N-1} I_i \phi^i
$$
whose coefficients $I_0,\ldots,I_{n-1}\in\CC$ of
give parameters for the space of Calogero-Moser spectral curves.

\bigskip
We are now ready to prove one direction of our main result, that the spectral curves of the elliptic Calogero-Moser system in fact lie in the locus $\calK_g$ defined using differentials with real periods:
\begin{prop}
For a fixed elliptic curve $E=E_\tau$ and a fixed integer $N$, the normalization $\widetilde C_{cm}$ of the spectral curve $C_{cm}$ of the Calogero-Moser system
lies in the Calogero-Moser locus $\calK_{g,N}^{\tau}\subset\calM_{g,1}$.
\end{prop}
\begin{proof}
Indeed, to prove this we need to construct two differentials $\Phi_1$ and $\Phi_2$ on $C_{cm}$ with all periods integer (and then also verify that the resulting $N$ is correct). To construct these differentials, we will think of the curve $C_{cm}^o\subset \CC\times(E\setminus\{0\})$, so that we can pull back differentials from both factors, and try to look for $\Phi_i$ of the form $f_i(k)dk +g_i(z) dz$. Note that since all periods of differentials on $\CC$ are zero, the first summand contributes nothing to the periods of $\Phi_i$, while the periods of the second summand are the periods of that differential on the image in $E$ of a cycle in $C$, and thus lie in $\ZZ+\tau \ZZ$. Thus we will take $g_i$ so that the two periods of the meromorphic differential $g_i(z)dz$ on $E$ are integer, and choose $f_i$ to ensure that the singularities are as required.

Indeed, when we compactify, the differential $dk$ has a double pole at $(\infty,0)\in\PP^1\times E$. As discussed above, this point has $N$ preimages on $C_{cm}$ corresponding to the $N$ local  branches near $(\infty,0)$, with local expressions given by \eqref{r1}, and the preimage of $(\infty,0)$ under the first blowup consists of two point $p$, with one branch through it, and $p'$ lying on $N-1$ branches. Since $dk$ has a double pole at $\infty\in\PP^1$, at all the $N-1$ branches where the coefficients $a_i=1$ in \eqref{r1}, so that the branch is locally given by $k+z^{-1}+h_i=O(z)$, we can cancel the double pole of $dk$ near $k=\infty$ by taking locally $-dz/z^2$; since we are working on the elliptic curve, this means we should globally take $-\wp(z)dz$, which precisely this double pole at $z=0$. Thus we are looking for $\Phi_i=a_i (dk -\wp(z)dz)+c_idz$ for some constants $a_i,c_i\in\CC$, where the constant $c_i$ is determined to ensure that the periods are integers. Solving we thus get
\begin{equation}\label{cm10}
 \Phi_1:=\frac{1}{2\pi i}(dk-\wp(z)dz)+c_1dz, \qquad\Phi_2:=\frac{\tau}{2\pi i}(dk-\wp(z)dz)+c_2dz.
\end{equation}
where
$$
  c_1:=\frac{1}{2\pi i}\int_0^1 \wp(z)dz\quad{\rm and}\quad c_2:=\frac{1}{2\pi i}\int_0^\tau \wp(z)dz.
$$
We have thus shown that the curve $C_{cm}$ indeed lies in $\calK_g$, and by construction the value of $N$ is as required.
\end{proof}

\bigskip
The other direction of the main result, theorem \ref{thm:main}, is the statement that any curve
$(C,p)\in \calK_{g,N}^{\tau}$, arises as the spectral curve $C_{cm}$ of the elliptic Calogero-Moser system.
This uses the methods of integrable systems: in \cite{kr1,kr2} the second author gave a general construction to obtain
an algebro-geometric solution of the KP equation starting from such a curve.
The statement that we need to prove is essentially that in the case when $\Phi_1$ and $\Phi_2$ both have
integral periods the solution of KP equation obtained in this way is elliptic.

To explain how this argument works, we recall the definition of the Baker-Akhiezer function and related constructions.



\begin{definition}
For $(C,p,z)\in\calM$, for a fixed and a generic set of $g$ points $\gamma_1,\ldots,\gamma_g\in C$
(that is, forming an effective divisor $Z_0:=\gamma_1+\ldots+\gamma_g$  of degree $g$ on $C$ with $h^0(C,D)=1$),
and for fixed $x,t\in\CC$ the Baker-Akhiezer function $\psi(x,t,p)$ is the unique
function on $C$, which is meromorphic on $C\setminus\{p\}$, with only singularities being the simple poles
at $\gamma_i$, and such that  in a neighborhood of $p$ it has an essential singularity that admits an expression
of the form
\begin{equation}\label{baexp}
\psi(x,t,z)=e^{xz^{-1}+tz^{-2}}\left(1+\sum_{s=1}^\infty \xi_s(x,t)z^{-s}\right)
\end{equation}
where each $\xi_s$ is some holomorphic function of $x,t$.
\end{definition}
The uniqueness of the Baker-Akhiezer function follows easily from observing that the ratio of two such functions would be
holomorphic on all of $C$, since the simple poles cancel, with value 1 at $p$, where the essential singularities
cancel. To see the existence of the Baker-Akhiezer function, note that an explicit expression for it was obtained in
\cite{kr1}:
\begin{equation}\label{psitheta}
\psi(x,t,q)=\frac{\theta (A(q)+Ux+Vt+Z_0)\,\theta (A(q)+Z_0)}{ \theta (A(q)+Ux+Vt+Z_0)\,
\theta (A(q)+Z_0)}\, e^{x\int^q\Omega_2+t\int^p\Omega_3},
\end{equation}
where $A:C\hookrightarrow J(C)$ is the Abel-Jacobi embedding of the curve into its Jacobian, and
$U$ and $V$ are the vectors of $B$-periods of the normalized (i.e.~with all $A$-periods zero) differentials $\Omega_2$ and
$\Omega_3$, with poles at $p$ of second and third order, respectively, and holomorphic elsewhere.

The construction of CM curves was crucial for the identification of the theory of the CM system and the theory of the elliptic solutions of the Kadomtsev-Petviashvili (KP) equation established in \cite{kr-cm}.
This identification is based on the following result:

\begin{lm}\label{genform}
The equation
\begin{equation}\label{lax11}
\left(\partial_t-\partial_x^2+u(x,t)\right)\psi(x,t)=0
\end{equation}
with elliptic potential (i.e. u(x,t) is an elliptic function of the variable $x$) has a meromorphic
in $x$ solution $\psi$ if and only if $u$ is of the form
\begin{equation}\label{u_ansatz}
u=2\sum_{i=1}^N \wp (x-x_i(t))
\end{equation}
with poles $x_i(t)$ satisfying the equations of motion of the CM system.
\end{lm}
\begin{rem}
In \cite{kr-cm} a slightly weaker form of the lemma was proven. Namely, its assertion was proved under the assumption that equation \eqref{u_ansatz} has a family of {\it double-Bloch} solutions (i.e. meromorphic solutions
with monodromy $\psi(x+\omega_\alpha,t)=w_\alpha \psi(x,t)$, where $\omega_a$ are periods of the elliptic curve and $w_a$ are constants.) This weaker version is sufficient for our further purposes, but
for completeness we included above the strongest form of the lemma, proven in \cite{flex} (see
\cite{kr-schot} for details).
\end{rem}

As shown in \cite{kr1}, the Baker-Akhiezer function satisfies partial differential equation
\eqref{lax11} with the potential $u(x,t)$ given explicitly as
\begin{equation}\label{utheta}
 u(x,t)=2\partial_x^2\ln \theta(Ux+Vt+Z_0).
\end{equation}

We will now use the Baker-Akhiezer function to obtain our main result.
\begin{proof}[Proof of main theorem \ref{thm:main}]
Starting from any curve $(C,p,z)$ (and a collection of $g$ points on $C$
in a general position) we can construct uniquely a Baker-Akhiezer function,
given explicitly by \eqref{psitheta}. We first show that the curves $(C,p)\in\calK_g$
are characterized within $\calM_{g,1}$ by the property that the vector $U$ in
\eqref{psitheta},\eqref{utheta} spans an elliptic curve in the Jacobian of $C$
(i.e. ~$\CC U\subset J(C)$ is closed).

Indeed, recall that $U$ is the vector of $B$-periods of the meromorphic differential
$\Omega_2$, which has a double pole at $p$, and all of which $A$-periods are zero.
For $(C,p)\in\calK_g$ we then have two holomorphic differentials $\Omega_2-\Phi_1$
and $\tau\Omega_2-\Phi_2$. Since all the $A$-periods of $\Omega_2$ are zero,
the $A$-periods of these two differentials are all
integer, and thus both $\Omega_2-\Phi_1$ and $\tau\Omega_2-\Phi_2$ must be linear combinations
of a basis $\omega_1\ldots \omega_g$ of holomorphic differentials on $C$ dual to the $A$-cycles,
with {\em integer} coefficients. Thus we have
$$
  \Omega_2=\Phi_1+w_1=(\Phi_2+w_2)/\tau
$$
where holomorphic differential $w_1,w_2$ are integral linear combinations of $\omega_i$.
From the first equality it follows that the vector $U$ of $B$-periods of $\Omega_2$ is equal
to the sum of $B$-periods of $\Phi_1$, which are integers, and the $B$-periods of $w_1$, which
are integral linear combinations of the periods of $\omega_i$. This means that we have
$U\in\ZZ^g+\tau_C\ZZ^g$, where $\tau_C$ is the period matrix of $C$
(the matrix of $B$-periods of $\omega_i$). Similarly from the second expression for $\Omega_2$
it follows that also $\tau U\in\ZZ^g+\tau_C\ZZ^g$. Finally since $\CC U=\RR U+\RR\tau U$ is
then a complex line containing two non-proportional vectors in the lattice $\ZZ^g+\tau_C\ZZ^g$,
its image in the Jacobian $J(C):=\CC^g/\ZZ^g+\tau_C\ZZ^g$ is compact, and thus $U$ spans
an elliptic curve in $J(C)$.

Let now $N$ be the degree of the restriction of the theta function from $J(C)$ to the elliptic
curve $E$ generated by $U$. Then the restriction of the theta function of $J(C)$ to $E$
can be written as a product in terms of its zeroes using the elliptic $\sigma$ function:
\begin{equation}\label{restriction}
\theta(\tau_C,Ux+Vt+Z_0)=f(t,Z_0)\prod_{i=1}^N \sigma(x-x_i(t))
\end{equation}
for $t$ and $Z_0$ fixed, where $f$ is non-zero (and of course depends on $t$ and $Z_0$
holomorphically), and $x_1(t),\ldots x_N(t)$ are the zeroes of the restriction of the theta
function of $J(C)$ to the corresponding translate of $E$.
Substituting this expression into \eqref{utheta} implies that $u$ is of the form \eqref{u_ansatz}. The Baker-Akhiezer function is a meromorphic function of $x$. {}From lemma \ref{genform} it then follows that
$q_i(t)$ in \eqref{restriction} satisfy the equations of motion of the CM system.
\begin{rem}
For the last statement a weaker version of lemma \ref{genform} suffices, because from the definition
of the Baker-Akhiezer function it follows that it has monodromy given by
\begin{equation}\label{doubleBloch}
\psi (x+1,t,p)=e^{2\pi i F_1(p)}\psi(x,t,p),\qquad
\psi(x+\tau,t,p)=e^{2\pi i F_2(p)}\psi(x,t,p).
\end{equation}
Such monodromy was called in \cite{kr-bete} the double-Bloch property.
Geometrically, it means that as a function of $x$ for $t$ and $p$ fixed, $\psi$ is a (meromorphic) section of a certain
bundle on the elliptic curve $E=\CC/\ZZ+\tau\ZZ$, where this bundle depends on $p$.
\end{rem}

Thus, starting from $(C,p)\in\calK_{g,N}$, we have used the Baker-Akhiezer function (which is a solution
of the KP system) to then construct a solution of the elliptic Calogero-Moser system. This elliptic Calogero-Moser system has a spectral curve given explicitly by equation
\eqref{r}. It thus remains to show that the $\widetilde C_{cm}$ of the spectral curve of this CM system indeed
coincides with the original curve $C$. To this end it is enough to check that
\begin{equation}\label{double-Bloch1}
 \psi(x,t,p)=\sum_{i=1}^N c_i(t,p) F(x-x_i(t),z)e^{kx+k^2t}
\end{equation}
satisfies all the defining properties of the Baker-Akhiezer function on $\widetilde C_{cm}$ --- and thus by uniqueness is the Baker-Akhiezer function.
Here the functions $c_i$ are coordinates of the vector $C=(c_1,\ldots,c_N)$ satisfying
\begin{equation}
  (L(t,z)+k)\,\, C=0,\quad \dot C=M(t,z)C\,.
\end{equation}
This verification is straightforward, and the proof is thus complete.
\end{proof}

Since we have constructed all curves in $\calK_g$ as normalizations of spectral
curves of the Calogero-Moser system, and normalizing can only reduce the arithmetic genus,
we get in particular
\begin{cor}
The locus $\calK_{g,N}^{\tau}$ is empty if $g>N$.
\end{cor}

\begin{rem}
We note that $\calK_{g,g}$ by construction is the locus of spectral curves that are smooth. indeed, for $C_{cm}\in\calK_{g,g}$, the differentials $\Phi_1$ and $\Phi_2$ cannot have common zeroes. If for some point  $p\in C_{cm}$ we had $\Phi_1(p)=\Phi_2(p)=0$, then also $dk(p)=dz(p)=0$, as these two differentials are linear combinations of $\Phi_1$ and $\Phi_2$. However, if both $dk$ and $dz$ vanish at a point of $\overline{C_{cm}^o}$, this point is singular, while we assumed the curve to be smooth. Following this line of thought, one would expect the common zeroes of $\Phi_1$ and $\Phi_2$ on Calogero-Moser curves to be closely related to the singularities of the curve. For the two simplest possible classes of singularities --- nodes and simple cusps --- the situation is as follows: the differentials $\Psi_1$ and $\Psi_2$ (or equivalently $dk$ and $dz$) do not have a zero at a point of $\widetilde C_{cm}$ that is a preimage of a node on $C_{cm}$, and have a simple common zero at a preimage of a cusp (and a multiple common zero at a preimage of any more complicated singularity). It can be shown that a Zariski open subset of $\calK_{g,N}$ corresponds to singular CM curves having $N-g$ nodes.
\end{rem}


\end{document}